\documentclass[12pt]{amsart}
\usepackage{amsmath}
\usepackage{amsfonts}
\usepackage{amssymb}
\usepackage{mathtools}

\usepackage{amsthm}
    \newtheorem{prop}{Proposition}
    \newtheorem{thm}{Theorem}
    
    \newtheorem{defn}{Definition}
    \newtheorem{cor}{Corollary}
\theoremstyle{remark}

\usepackage{xcolor}
\usepackage{fullpage}
\usepackage{enumitem}
\usepackage{bm} 

\allowdisplaybreaks

\title{Classifying matrix-valued holomorphic cross-sections over an annulus up to complete isometric isomorphism}
\author{Jacob Cornejo and Kathryn McCormick }
\address{Dept of Math and Stats, CSU Long Beach, Long Beach, CA 98040, USA}


\begin{document}

\maketitle

\begin{abstract}
 We classify certain algebras of matrix-valued cross-sections over an annulus up to complete isometric isomorphism, based on topological bundle invariants. In particular, we study sections of matrix bundles which are continuous on the closure of the annulus and holomorphic on its interior. Our strategy includes exploiting the relationship between concomitants and modulus automorphic functions, as well as the classification of $n$-homogeneous $C^*$-algebras by Fell and Tomiyama-Takesaki. Furthermore, we describe a partial extension of our results over the annulus to larger classes of finitely and smoothly bordered planar domains.
    \end{abstract}

\section{Introduction} 
Let $R$ be a bounded domain in the complex plane with a boundary consisting of finitely many disjoint analytic Jordan curves. On $R$, we can define analogues of the disk algebra and of Hardy spaces $H^p(R)$ \cite{Rudin55}. In studying these algebras' function theory, invariant subspaces, and related operator theory, one naturally encounters algebras of multiple-valued functions on the disk which have single-valued modulus, i.e.~modulus-automorphic functions \cite{Sarason1965,Voichick1964}. These multiple-valued function algebras can also be realized as cross-sections of line bundles over $R$ \cite{AbrahamseDouglas76,Widom71}.

More specifically, in \cite{Sarason1965} Sarason describes the function theory and invariant subspaces of $H^p$ spaces on the annulus, and independently in \cite{Voichick1964} Voichick studies invariant subspaces for $H^2$ on finite Riemann surfaces. They attack these problems by lifting their functions on $R$ to the universal covering space of their Riemann surfaces, the disk. Their study of the invariant subspaces for modulus-automorphic functions on the disk has been followed up in work of many others, such as \cite{Hasumi1966}.

Sarason shows that the bounded linear operators on $H^2$-spaces of modulus-automorphic functions that are operators that commute with the shift are the operators that can be implemented by multiplication by bounded modulus automorphic functions. He then uses this result to describe when two shifts are unitarily equivalent \cite[Thm.~10, Cor.~1]{Sarason1965}. When one translates his work over to the line bundle perspective, we have a classification up to unitary equivalence of the shift operators, viewed as operators on sections of line bundles, based on a line bundle invariant. This second perspective is even more fully developed in the theory of bundle shifts in \cite{AbrahamseDouglas76}, where Abrahamse and Douglas now consider vector bundles over more general planar domains than the annulus. 

One can ask, though, instead of classifying individual shift operators based on a bundle invariant, can we classify operator algebras that contain shift operators based on a bundle invariant? In particular, we can study the algebra of cross-sections of an appropriately chosen matrix bundle, which can act on the cross-sections of a vector bundle.

The purpose of this paper is to classify, up to complete isometric isomorphism, the algebra of $M_n(\mathbb{C})$-valued cross-sections over the annulus which are continuous sections on the closed annulus and holomorphic sections on the interior. We show that these algebras are classified by a bundle invariant. We will also briefly describe how to generalize our results to finite and smoothly bordered domains in the complex plane in the special case where the associated bundle is determined by commuting unitary matrices. 

This work can also be seen as a  nonselfadjoint analogue to Fell \cite{Fell1961} and Tomiyama-Takesaki's \cite{TomiyamaTakesaki61} classification of $n$-homogeneous $C^*$-algebras, and we use this $n$-homogeneous classification as one ingredient in our proof. (The matrix section algebras are even \emph{completely isometrically $n$-subhomogeneous}, using the terminology of \cite{AHMR2022}.) These nonselfadjoint algebras also provide another family of unital operator algebras that are not uniform algebras and have the Bishop Property of \cite{ClouatreThompson2023} inside their natural $C^*$-superalgebra.

\section{Preliminaries}
We let $\mathbb{C}$ and $M_n(\mathbb{C})$ denote the complex numbers and $n \times n$ matrices over the complex numbers, respectively, and $I_n$ will denote the $n\times n$ identity matrix. Then $U_n(\mathbb{C})$ is the collection of unitary $n \times n$ matrices, and $PU_n(\mathbb{C})$ denotes the projective unitary group. Recall that the elements of $PU_n(\mathbb{C})$ can be identified with conjugation by a particular unitary matrix. Let $\text{Ad}(A): M_n(\mathbb{C}) \to M_n(\mathbb{C})$ denote the map $X \mapsto A X A^{-1}$; for $A \in U_n(\mathbb{C})$, this is equivalent to considering the map $X \mapsto A X A^*$ where $^*$ denotes the complex conjugate transpose. The elements of $PU_n(\mathbb{C})$ can be identified (non-uniquely) as $\text{Ad}(A)$ for some matrix $A \in U_n(\mathbb{C})$. We will use the notation $[n]:=\{1,2,\ldots, n\}$ for brevity.  

\subsection{Topological background} \label{subsection:topnot}

For most of the narrative, $R$ will specifically denote an open annulus in the complex plane with inner radius $r_0>0$ and outer radius $r_1>r_0$. Two open annuli are conformally equivalent if and only if they have the same ratio $r_1/r_0$, thus we assume throughout that $r_1>r_0=1$. In Section \ref{sec:GeneralizationToOtherDomains}, $R$ will be used to more generally denote a planar domain whose boundary $\partial R$ is piecewise-smooth and consists of finitely many disjoint simple closed curves with $\overline{R}=R \cup \partial R$; we will sometimes refer to a `general $R$' to indicate a result generalizes to this setting.

We let $A(\overline{R})$ denote the collection of functions $f : \overline{R} \to \mathbb{C}$ such that $f$ is continuous on $\overline{R}$ and holomorphic on $R$; we call $A(\overline{R})$ the algebra of continuous-holomorphic functions over $\overline{R}$. For general $R$, $A(\overline{R})$ is generated as a uniform algebra by the collection of rational functions with poles off of $\overline{R}$ \cite[Thm.~2.3]{Mergelyan54}.
The algebra of continuous $\mathbb{C}$-valued functions on $\overline{R}$  will be denoted $C(\overline{R})$. When $R$ is specifically an annulus, the algebras $C(\overline{R})$ and $A(\overline{R})$ were studied extensively in Sarason's thesis \cite{Sarason1965}. For general $R$, the algebras have been studied in, for example, \cite{Hasumi1966, Voichick1964,Stanton1981}.

We will build a noncommutative version of $A(\overline{R}) \subseteq C(\overline{R})$ by considering cross-sections of a matrix bundle $\mathcal{B}$ over $\overline{R}$, in such a way that the center of the algebra of continuous sections will be $C(\overline{R})I_n$. To accomplish this, we will ask that $\mathcal{B}$ be continuous vector bundle with matrix fibres $M_n(\mathbb{C})$. We also require that $\mathcal{B}$ has a holomorphic structure when restricted to the interior $R$, so that the center of the algebra of continuous sections holomorphic on the interior of the bundle is $A(\overline{R})I_n$.

For our algebra of continuous sections to be a $C^*$-algebra, we need to require that $\mathcal{B}$ is a bundle with transition functions which are $PU_n(\mathbb{C})$-valued. To ask that $\mathcal{B}$ have holomorphic structure on the interior and have projective unitary-valued transition functions, $\mathcal{B}$ must have locally constant transition functions -- in other words, we will assume that $\mathcal{B}$ is a flat bundle.

Therefore, we let $\mathcal{B}$ to be a flat $PU_n(\mathbb{C})$-bundle with matrix fibres, which is then determined by a representation $\rho$ of the fundamental group of $\overline{R}$ \cite{Gunning1967}. We sometimes write $\mathcal{B}=\mathcal{B}(\overline{R}, \rho)$ to denote the dependence on $\rho$. In our main case, where $\overline{R}$ is an annulus, $\pi_1(\overline{R}) \simeq \mathbb{Z}$, $\rho: \mathbb{Z} \to PU_n(\mathbb{C})$, and $\rho$ is completely determined by $\rho(1)$. Pick an open cover $\mathcal{O}$ of $\overline{R}$, and let $U,V \in \mathcal{O}$. The bundle $\mathcal{B}$ has transition functions $g_{UV}: U \cap V \to PU_n(\mathbb{C})$, where $g_{UV}(w) \in PU_n(\mathbb{C})$ acts on the fibres $M_n(\mathbb{C})$ by conjugating by powers of a fixed matrix in $U_n(\mathbb{C})$. 

In a particular category of bundles and bundle maps, we can ask that (1) all the bundles have the same base space and that the bundle maps fix the base space, or that (2) the bundle maps are allowed to transform the base space by a homeomorphism or (2') the bundle maps are allowed to transform the base space by a homeomorphism that is holomorphic on the interior. Suppose the two bundles in question have transition functions $g_{UV}$, $U,V \in \mathcal{O}$ and $h_{U'V'}$, $U',V' \in \mathcal{O'}$. In the first sense, a bundle equivalence is implemented by choosing a refinement $\mathcal{O''}$ of the open covers $\mathcal{O}, \mathcal{O'}$, and finding functions $j_{U''}: U'' \to PU_n(\mathbb{C})$ such that $h_{U''V''} = j_{U''} g_{U''V''} j_{V''}^{-1}$ for every $U'',V'' \in \mathcal{O}''$. In the second sense, a bundle equivalence is implemented in the same way but by additionally composing with a homeomorphism; or in the third sense, by additionally composing with a homeomorphism that is holomorphic on the interior.

If we say our bundle equivalence is a \emph{flat} equivalence of (flat) bundles $\mathcal{B}(\overline{R}, \rho)$ and $\mathcal{B}(\overline{S}, \tau)$, we mean that the functions $j_{U''}: U'' \to PU_n(\mathbb{C})$ implementing the equivalence can be chosen to be (locally) constant. From the representation perspective, this is equivalent to (1) requiring the representations $\rho$ and $\tau$ defining the bundles are $PU_n(\mathbb{C})$-conjugate and $\overline{R}=\overline{S}$, or (2) requiring that $\rho$ and $\tau$ are conjugate and that there is a homeomorphism between $\overline{R}$ and $\overline{S}$, or (2') requiring that $\rho$ and $\tau$ are conjugate and that there is a homeomorphism between $\overline{R}$ and $\overline{S}$ that is holomorphic on the interior. 

Fix representations $ \rho, \tau: \mathbb{Z} \to PU_n(\mathbb{C})$ that determine the flat bundles $\mathcal{B}(\overline{R}, \rho)$, $\mathcal{B}(\overline{R},\tau)$ over an annulus. These representations, in turn, are determined by matrices $A_\rho, A_\tau \in U_n(\mathbb{C})$ for which $\rho(1)= \text{Ad}(A_\rho)$ and $\tau(1)=\text{Ad}(A_\tau)$. The representations $\rho$ and $\tau$ will be $PU_n(\mathbb{C})$-conjugate if and only if the representative matrices $A_\rho, A_\tau$ are unitarily conjugate. 
	
If we restrict the bundle $\mathcal{B}=\mathcal{B}(\overline{R}, \rho)$ to $\partial R$, then we will have a continuous, flat bundle $\mathcal{B}(\overline{R}, \rho)|_{\partial R}$, and if we restrict to the interior of $\overline{R}$, we will have a holomorphic, flat bundle $\mathcal{B}(\overline{R}, \rho)|_{R}$. The restriction of a flat bundle over $\overline{R}$ to a bundle over $S \subseteq \overline{R}$ will also be a flat bundle by simply restricting the locally constant transition functions. However, its equivalence class as a bundle may be very different, as the domain of functions implementing bundle equivalence has been changed. Also, it somewhat complicates the perspective of representations of the fundamental group since the base space may no longer be connected. However, the restriction to boundary components or the interior is sufficiently nicely behaved in our setting to make some identifications. Choose a point $w_0 \in \text{int}(R)$, and points $w_1, w_2, \ldots, w_b \in \partial R$ for each boundary component. If we identify $\pi_1(\overline{R})=\pi_1(\overline{R},w_0)$, restricting to $\text{int}(R)$ we will get $\pi_1(\overline{R},w_0) \simeq \pi_1(R, w_0)$ in such a way that preserves the relationship between flat bundles and representations of the fundamental group. In the case of restricting to the boundary, it is certainly true that $\pi_1(\overline{R},w_0) \simeq \pi_1(\overline{R},w_i)$ ($i \not = 0$) via a choice of path from $w_0$ to $w_i$. Fix such a system of paths, and a representation of $\pi_1(\overline{R},w_0)$ coming from a flat bundle structure. If we restrict $\overline{R}$ to $\partial R$, and look at the fundamental groupoid $\Pi_1(\partial R, w_1, \ldots, w_b)$ of $\partial R$ restricted to the points $w_1, \ldots, w_b$, then representations of $\Pi_1(\partial R, w_1, \ldots, w_b)$ can be identified with a product of representations of each connected component. By choosing base points and restricting to $\partial R$ we induce a representation of $\Pi_1(\partial R):=\Pi_1(\partial R, w_1, \ldots, w_b)$ into $PU_n(\mathbb{C})$, which we call the \emph{$\rho$-induced representation}. For example, suppose $\overline{R}$ is the closed annulus and $\rho(1)=\text{Ad}(A)$ for some $A \in U_n(\mathbb{C})$. The fundamental groupoid of $\partial R$ restricted to two basepoints $w_1, w_2$ on the boundary circles is a trivial $\mathbb{Z}$-bundle over two points, and the $\rho$-induced representation can be identified with the homomorphism $\rho' : \mathbb{Z} \times \mathbb{Z} \to PU_n(\mathbb{C})$ satisfying $\rho'(1,0)=\rho'(0,1)=\text{Ad}(A)$. 

The following definitions clarify what we mean by flat equivalence for bundles that have been restricted to the boundary of $\overline{R}$. The goal will be to classify an algebra of cross-sections based on the bundle's equivalence class.

\begin{defn} Consider a flat $PU_n(\mathbb{C})$ bundle $\mathcal{B}(\overline{R}, \rho)$ over $\overline{R}$ and a flat $PU_n(\mathbb{C})$ bundle $\mathcal{B}(\partial R, \tau)$ over $\partial R$. 
\begin{enumerate}
    \item  The bundle $\mathcal{B}(\overline{R}, \rho)|_{\partial R}$ is called \emph{restricted flat equivalent to $\mathcal{B}(\partial R, \tau)$ in the sense (1)} if the bundles are flat equivalent, the base spaces are fixed by the bundle equivalence, and the $\rho$-induced representation of $\Pi_1(\partial R)$ and the representation $\tau$ of $\Pi_1(\partial R)$ are conjugate by a (single, fixed) matrix. 
    \item The bundle $\mathcal{B}(\overline{R}, \rho)|_{\partial R}$ is called \emph{restricted flat equivalent to $\mathcal{B}(\partial R, \tau)$ in the sense (2)} if the bundles are flat equivalent; the $\rho$-induced representation of $\Pi_1(\partial R)$ and the representation $\tau$ are conjugate by a (single, fixed) matrix; and there is a homeomorphism $\phi : \partial R \to \partial R$ in the bundle equivalence in sense (2) such that the homeomorphism is the restriction of a homeomorphism of $\overline{R}$ that is holomorphic on the interior $R$.
\end{enumerate}
    \end{defn}

\subsection{Functional analysis preliminaries} \label{subsection:funcprelim}
For references on operator spaces, completely bounded maps, operator algebras, and the $C^*$-envelope, see for example \cite{Paulsen2002,BlM2004}. We will also make use of the theory of $n$-homogeneous $C^*$-algebras, which can be found in \cite{TomiyamaTakesaki61,Fell1961}.

Recall $A(\overline{R})$ is used to denote the $\mathbb{C}$-valued continuous-holomorphic functions defined on $\overline{R}$ ,  $\Gamma_c(\overline{R},\mathcal{B}(\overline{R}))$ denote the $C^*$-algebra of continuous cross-sections of $\mathcal{B}(\overline{R})$, and $\Gamma_{h}(\overline{R},\mathcal{B}(\overline{R}))$ denote the holomorphic subalgebra of $\Gamma_c(\overline{R},\mathcal{B}(\overline{R}))$ where the sections are holomorphic on the interior of $\overline{R}$ and continuous on the boundary.

There are two ways we will view a cross-section $\sigma$ of the bundle $\mathcal{B}$. A continuous section $\sigma$ is defined locally by picking an open cover $\mathcal{O}$ of $\overline{R}$, and then $\sigma= (\sigma_U)_{U \in \mathcal{O}}$ where $\sigma_U : U \subseteq R \to M_n(\mathbb{C})$ is a continuous function and the ambiguities of defining $\sigma_U$ and $\sigma_V$ where $U \cap V$ is nonempty are resolved via the transition functions. However, since our bundle $\mathcal{B}$ is flat, we can also lift $\sigma$ to a well-defined function on the universal covering space of $\overline{R}$. We denote the universal covering space of $\overline{R}$ as $\widetilde{D}$, viewed as the union of an open disk $D$ with its boundary components, and for which the complement of the boundary forms a Cantor set in the boundary circle. Then $\sigma$ may be identified with a function $F_\sigma: \widetilde{D} \to M_n(\mathbb{C})$ that satisfies $F_\sigma( g \cdot x) = \rho(g) \cdot F_\sigma(x)$.  This perspective of $\sigma$ is useful for many of our arguments.

In either perspective, we can put a norm on the continuous sections by using the standard $C^*$-norm on the fibres $M_n(\mathbb{C})$ and taking $\| \sigma \|= \sup_{z \in \overline{R}} \| \sigma_U (z)\|= \sup_{w \in \tilde{D}} \| F_{\sigma}(w)\|$, which is well-defined. Since the collection of continuous sections of $\mathcal{B}$ is a $C^*$-algebra with respect to this norm and the standard involution, we can put the canonical operator space structure on $\Gamma_h(\overline{R},\mathcal{B}(\overline{R}))$. The collection of continuous sections that are holomorphic on the interior form a non-self-adjoint, closed subalgebra of the $C^*$-algebra, and inherit the superalgebra's operator space structure. See \cite{McCormick19} for more details. 

The algebra $A(\partial R) \subseteq C(\partial R)$ is the algebra of continuous functions on $\partial R$ that are restrictions of functions in $A(\overline{R})$, i.e.~$A(\partial R)$ consists of continuous functions that extend to be holomorphic on $R$. Since the norm is attained on the boundary, $A(\partial R) \simeq A(\overline{R})$. Similarly, $\Gamma_h(\partial R,\mathcal{B}(\overline{R})|_{\partial R}) \simeq \Gamma_h(\overline{R},\mathcal{B}(\overline{R}))$ \cite[Lem.~3.3]{McCormick19}. The algebra $A(\overline{R})$ can also be viewed as the uniform closure of rational functions with poles off of $R$ \cite{Kodama1965}.

\section{Results} \label{section:results}

The first proposition is true for general $R$:

\begin{prop} \label{propeasydirection} Suppose that $\mathcal{B}(\overline{R},\rho)|_{\partial R}$ is restricted flat $PU_n(\mathbb{C})$ equivalent to the bundle $\mathcal{B}(\overline{R},\tau)|_{\partial R}$ in the sense (2). Then $\Gamma_h(\overline{R},\mathcal{B}(\overline{R},\rho))$ is completely isometrically isomorphic to $\Gamma_h(\overline{R},\mathcal{B}(\overline{R},{\tau}))$. If the bundle equivalence fixes $\partial{R}$ pointwise (i.e.~is in the sense (1)), then the isomorphism of the section algebras pointwise-preserves the elements $f \in A(\overline{R})I_n$ in the center. 
    \end{prop}
\begin{proof} For the first claim, it's enough to show that $\Gamma_h(\partial{R},\mathcal{B}(\overline{R},\rho)|_{\partial R}) \simeq \Gamma_h(\partial{R},\mathcal{B}(\overline{R},{\tau})|_{\partial R})$, since the algebra $\Gamma_h(\overline{R},\mathcal{B}(\overline{R},\rho)) \simeq \Gamma_h(\partial{R},\mathcal{B}(\overline{R},\rho)|_{\partial R}) $ via the restriction map.

Suppose $\mathcal{B}(\overline{R},\rho)|_{\partial R}$ is restricted flat $PU_n(\mathbb{C})$ equivalent to the bundle $\mathcal{B}(\overline{R},\tau)|_{\partial R}$ in sense (2), then there is a homeomorphism $\varphi: \partial{R} \to \partial{R}$ and a matrix $V\in U_n(\mathbb{C})$ so that $\tau(k) = V^{-1} \rho(k) V$ for all $k \in \pi_1(\overline{R})$. Let $\tilde{\varphi}$ be the induced map by $\varphi$ on the universal covers of $\partial{R}$ and $\varphi(\partial{R})$ which satisfies $\tilde{\varphi}(k \cdot z)= k \cdot \tilde{\varphi}(z)$.

View $\Gamma_h(\partial{R},\mathcal{B}(\overline{R},\rho)|_{\partial R})$ as the collection of restrictions of (continuous holomorphic) concomitants $F: \partial\widetilde{D} \to M_n(\mathbb{C})$ satisfying $F(k \cdot z)=\rho(k) \cdot F(z)$ for every $k \in \pi_1(\overline{R})$ and $z \in \partial\widetilde{D}$. Also view $\Gamma_h(\partial R,\mathcal{B}(\overline{R},\tau)|_{\partial R})$ as the collection of concomitants $F_1: \tilde{\varphi}(\partial\widetilde{D}) \to M_n(\mathbb{C})$ satisfying $F_1(k \cdot z_1)=\tau(k) \cdot F_1(z_1)$ for every $z_1 \in \tilde{\varphi}(\partial\widetilde{D})$, or in other words, for every $z \in \partial\widetilde{D}$,

\begin{align*}F_1(\tilde{\varphi}(k \cdot z)) & =\tau(k) \cdot F_1(\tilde{\varphi}(z)) \\
& = V^{-1} \rho(k) V \cdot F_1(\tilde{\varphi}(z)) \\
& = V^{-1} (\rho(k) \cdot (V F_1(\tilde{\varphi}(z)) V^{-1} ) )V .
\end{align*}

So, $V(F_1(\tilde{\varphi}(\cdot)))V^{-1} : \partial\tilde{D} \to M_n(\mathbb{C})$ is a concomitant with respect to $\rho$. Since $\varphi$ is the restriction of a holomorphic homeomorphism of $\overline{R}$, we conclude $V(F_1(\tilde{\varphi}(\cdot)))V^{-1}$ belongs to $\Gamma_h(\partial R, \mathcal{B}(\overline{R},\rho)|_{\partial R})$. Moreover, the norm of $V(F_1(\tilde{\varphi}(\cdot)))V^{-1}$ at any matrix level is equal to the norm of $F_1(\tilde{\varphi}(\cdot))$. Therefore, the map $F_1 \mapsto V(F_1(\tilde{\varphi}(\cdot)))V^{-1}$ a complete isometry, and which is also quickly seen to be an isomorphism of $\Gamma_h(\partial R, \mathcal{B}(\overline{R},\tau)|_{\partial R})$ with $\Gamma_h(\partial R, \mathcal{B}(\partial R, \overline{R},\rho)|_{\partial R})$.

Suppose that in addition, $\tilde{\varphi}(z)=z$. Then the isomorphism is defined as 
\[F_1 \mapsto V(F_1(\cdot))V^{-1}\]
If we use $F_1(z)=f(z)I_n$ for some $f \in A(\partial{R})$, then we will have
\[f(z)I_n \mapsto V(f(z)I_n)V^{-1}= f(z)I_n. \]
    \end{proof}

Theorem \ref{mainthm} will follow from the statement:

\begin{prop}\label{mainthmmod} Let $R$ be an annulus. Suppose that $\Gamma_h(\overline{R}, \mathcal{B}(\overline{R}, \rho))$ is completely isometrically isomorphic to $\Gamma_h(\overline{R}, \mathcal{B}(\overline{R}, \tau))$ and that the isomorphism takes $f \in A(\overline{R})I_n \subseteq \Gamma_h(\overline{R}, \mathcal{B}(\overline{R}, \rho)) \mapsto f \in A(\overline{R})I_n \subseteq \Gamma_h(\overline{R}, \mathcal{B}(\overline{R}, \tau))$ pointwise. Then $\mathcal{B}(\overline{R}, \rho)|_{\partial R}$ is restricted flat $PU_n(\mathbb{C})$ bundle equivalent to $\mathcal{B}(\overline{R},\tau)|_{\partial R}$ such that the bundle equivalence fixes $\partial {R}$ pointwise.
    \end{prop}
For, suppose that $\Gamma_h(\overline{R}, \mathcal{B}(\overline{R}, \rho))$ is completely isometrically isomorphic to $\Gamma_h(\overline{R}, \mathcal{B}(\overline{R}, \tau))$ via an isomorphism $\Phi$, but that $\Phi$ doesn't necessarily fix the center pointwise. The isomorphism on sections will induce an isometric isomorphism on the center algebras $\phi : A(\overline{R}) \to A(\overline{R})$. An automorphism of $A(\overline{R})$ induces a bijective map $t: \overline{R} \to \overline{R}$ via $t:=\phi(z)$, and the functions $\phi$ and $\text{pre}_t: f \mapsto f\circ t$ agree on rational functions. By \cite[Cor.~4]{Kodama1965}, $A(\overline{R})$ is equal to the uniform closure of rational functions with poles off of $\overline{R}$, and so $\phi$ and $\text{pre}_t$ agree everywhere. 
Consider the new map $\sigma \mapsto \Phi(\sigma \circ t^{-1})$, which is an isomorphism that fixes $A(\overline{R})I_n$ pointwise, and then apply Prop \ref{mainthmmod}.

Thus from now on, by `flat equivalence' we always mean flat equivalence that fixes the base space $\overline{R}$ pointwise. \\

To prove Proposition \ref{mainthmmod}, it remains to show that if  $\Gamma_h(\overline{R},\mathcal{B}(\overline{R}, \rho))$ is completely isometrically isomorphic to $\Gamma_h(\overline{R},\mathcal{B}(\overline{R}, \tau))$ (where the isomorphism is identity on the centers), then $\mathcal{B}(\overline{R}, \rho)|_{\partial R}$ is restricted flat $PU_n(\mathbb{C})$ equivalent to $\mathcal{B}(\overline{R}, \tau)|_{\partial R}$. The outline of the proof is as follows:

 \begin{enumerate}[label=(\arabic*)]
	\item Compute several families of continuous holomorphic sections of $\mathcal{B}(\overline{R}, \rho)$. \label{proofstep1-computingsectionfamilites}
	\item Assume there is a complete isometric isomorphism of the continuous holomorphic section algebras. 
	\item Lift the complete isometric isomorphism to a $C^*$-isomomorphism of the $C^*$-envelopes which will restrict appropriately on the subalgebras. \label{proofstep3-liftisomorphism}
	\item Use results of Fell and Tomiyama and Takesaki to conclude that the $C^*$-isomorphism is induced by a topological $PU_n(\mathbb{C})$-bundle equivalence of the bundles $\mathcal{B}(\partial R, \rho)$ and $\mathcal{B}(\partial R, \tau)$.
	\item Explictly map the continuous holomorphic sections of $\mathcal{B}(\partial R, \rho)$ to sections of $\mathcal{B}(\partial R, \tau)$. \label{proofstep5-mapoversections}
	\item Use the previous item to conclude that the $PU_n(\mathbb{C})$ equivalence $(\lambda_U)_\mathcal{U}$ corresponds to an identically constant function on $\partial R$. \label{proofstep6-equivalenceisaconstant}
	\item Conclude $\rho$ is unitarily conjugate to $\tau$. \label{proofstep7-unitaryconjugacyend}
		\end{enumerate}

\subsection{Computing holomorphic sections} \label{section:constructingsections}

Recall from Section \ref{subsection:funcprelim} that, viewing $\widetilde{D}$ as the infinite strip, there is a one-to-one correspondence between continuous holomorphic concomitant functions $F: \widetilde{D} \to M_n(\mathbb{C})$ and sections $(\sigma_U : U \subseteq \overline{R} \to M_n(\mathbb{C}))_{U \in \mathcal{O}}$ for some open cover $\mathcal{O}$ of $\overline{R}$ and $U \in \mathcal{O}$, given by the maps:
\[ F \mapsto (\sigma_{F,U})_{U \in \mathcal{O}}\]
\[ (\sigma_U)_{U \in \mathcal{O}} \mapsto F_\sigma \]
where 
\[ \sigma_{F,U}(w)= F(\ln_U(w)) \text{ and}\]  
\[F_\sigma (z) = \sigma_U(e^{z}) \text{ whenever } e^z \in U.\] 
Thefore to produce examples of sections we will explicitly write down some concomitant functions, and then apply the first correspondence in the map above. We will also use the definition of a concomitant to make observations about its entries and the relationship with invariant and modulus automorphic functions described in \cite{Sarason1965}.

Fix a representation $\rho$ and a matrix $A \in U_n(\mathbb{C})$ such that $\rho(1)=\text{Ad}(A)$. If $A$ is also diagonal, we write $A=\text{diag}(a_1,\dots, a_n)$ for some $a_i \in \mathbb{T}$.  Given such a matrix, we may pick $K_{rs}$ to be a real number such that $e^{2\pi i K_{rs}} = \frac{a_{s}}{a_{r}}$ for $r < s$, $r,s\in[n]$, and $K_{rs} = K_{sr}$. This provides us with $\binom{n}{2}$ of these $K_{rs}$.

Next we will define two families of functions dependent on parameters $D_i, C_{jk} \in \mathbb{C}$, $i,j,k \in [n]$. These functions are clearly continuous on $\widetilde{D}$ and holomorphic on $D$: the constant function

 \begin{equation} \label{eqn:type1concomit} F_{n,\mathbf{D}}(z)=F_{D_1,D_2, \ldots, D_n}(z):=\begin{pmatrix} D_{1} & 0 & \cdots & 0 \\ 0 & D_2 & \ddots & 0 \\ \vdots & \ddots & \ddots & \vdots \\ 0 & \cdots &  0 & D_n\end{pmatrix}\end{equation} 
and the function $F_{n,\mathbf{C}}(z)\coloneqq (C_{jk}e^{\pm K_{jk}z})_{j,k \in [n]}$ where $C_{jk} \equiv 0$ if $j=k$ and $\pm$ is chosen to be $-$ for $j>k$ and $+$ for $j<k$, that is,

\begin{equation} \label{eqn:type2concomit} F_{n,\mathbf{C}}(z)= \begin{pmatrix} 
0 & C_{12}e^{K_{12}z} & \cdots & C_{1n}e^{K_{1n} z} \\ 
C_{21}e^{-K_{21}z} & 0 & \ddots & C_{2n}e^{K_{2n}z} \\ 
\vdots & \ddots & \ddots & \vdots \\ 
C_{n1}e^{-K_{n1}z} & \cdots & C_{n(n-1)}e^{-K_{n(n-1)}z} & 0
\end{pmatrix} .\end{equation} 

We can quickly verify that when $A$ is diagonal, the constant functions of type (\ref{eqn:type1concomit}) satisfy the concomitant condition with respect to $\rho$ if the matrix $A$ is diagonal. We can also verify that functions of type (\ref{eqn:type2concomit}) are also concomitants with respect to $\rho$ when $A$ is diagonal:

\begin{align*}
    F_{n,\textbf{C}}(z+2\pi i k ) & = \begin{pmatrix} 
        0 & C_{12}e^{K_{12}(z+2\pi i k) } & \cdots & C_{1n}e^{K_{1n}(z+2\pi ik) } \\ 
        C_{21}e^{-K_{21}(z+2\pi ik)} & 0 & \ddots & C_{2n}e^{K_{2n}(z+2\pi ik) } \\ 
        \vdots & \ddots & \ddots & \vdots \\ 
        C_{n1}e^{-K_{n1}(z+2\pi ik)} & \cdots & C_{n(n-1)}e^{-K_{n(n-1)}(z+2\pi ik) } & 0
        \end{pmatrix} \\
    & = \text{Diag}(a_i^k)
       \cdot \begin{pmatrix} 
        0 & C_{12}e^{-K_{12}z}  & \cdots & C_{1n}e^{-K_{1n}z} \\ 
        C_{21}e^{K_{21}z} & 0 & \ddots & C_{2n}e^{-K_{2n}z} \\ 
        \vdots & \ddots & \ddots & \vdots \\ 
         C_{n1}e^{K_{n1}z} & \cdots &  C_{n(n-1)}e^{K_{n(n-1)}z}  & 0
        \end{pmatrix} 
        \cdot \text{Diag}(a_i^{-k})
        \\
    & = A F_{n,\mathbf{C}}(z)A^{-k} = \rho(k) \cdot F_{n,\mathbf{C}}(z)
\end{align*}

By focusing on the entries of $F_{n,\textbf{D}}$ and $F_{n,\textbf{C}}$, the following observations can be made. First, the family $F_{n,\textbf{D}}$ could be extended more generally to any diagonal matrix of continuous holomorphic scalar-valued functions $D_i(z)$ which are invariant under the action of the fundamental group, namely that $D_i(z+2\pi i k) = D_i(z)$. Second, the entries of $F_{n, \mathbf{C}}(z)$ are examples of what Sarason calls \emph{modulus automorphic functions} \cite[I.3]{Sarason1965}: letting $f_{jk}(z):=C_{jk}e^{\pm K_{jk}z}$, for every $z \in \widetilde{D}$, $|f_{jk}(z)|=|f_{jk}(z+2\pi i k)|$; and moreover there is a constant $(\frac{a_k}{a_j})^{\pm 1}$ called the \emph{multiplier} for which $f_{jk}(z+2\pi i q)=(\frac{a_k}{a_j})^{\pm q} \, f_{jk}(z)$ for all $z \in \widetilde{D}, q \in \mathbb{Z}$. In fact, any modulus automorphic function can be written as the product of an invariant function with a power function \cite[pg.~12]{Sarason1965}. 

More generally, the following is true: let $\widetilde{D}$ be the universal covering space of a finitely, smoothly bordered planar domain $\overline{R}$. Suppose $\rho: \pi_1(\overline{R}) \to PU_n(\mathbb{C})$ is such that the image of the generators of $\pi_1(\overline{R})$ can be written as $\text{Ad}(A_i)$ where $A_i \in U_n(\mathbb{C})$ are diagonal. Then any entry of \emph{any} $\rho$-concomitant function is a modulus automorphic function. The diagonal entries will be invariant functions, i.e. modulus automorphic with multiplier 1. Note that this correspondence between concomitants and modulus automorphic functions does not follow through when the matrix $A$ is not diagonal.

\subsection{Mapping holomorphic sections, and conclusions about the bundle equivalence}

\begin{thm} \label{mainthm} Suppose that $\Gamma_h(\overline{R},\mathcal{B}(\overline{R},\rho))$ is completely isometrically isomorphic to $\Gamma_h(\overline{R},\mathcal{B}(\overline{R},\tau))$. Then $\mathcal{B}(\overline{R}, \rho)|_{\partial R}$ is restricted flat $PU_n(\mathbb{C})$ equivalent to $\mathcal{B}(\overline{R}, \tau)|_{\partial R}$.
    \end{thm}

\begin{proof} Let $\Phi : \Gamma_h(\overline{R},\mathcal{B}(\overline{R},\rho)) \to \Gamma_h(\overline{R},\mathcal{B}(\overline{R},\tau))$ be a complete isometric isomorphism that preserves the center pointwise. Then $\Phi$ lifts to a complete isometric isomorphism $\Psi$ of the $C^*$-envelopes; by \cite{McCormick19}, this means $\Psi : \Gamma_c(\partial R, \mathcal{B}(\overline{R}, \rho)|_{\partial R}) \to \Gamma_c(\partial R, \mathcal{B}(\overline{R}, \tau)|_{\partial R})$ where the image of $\Gamma_h(\overline{R},\mathcal{B}(\overline{R},\rho))$ in $\Gamma_c(\partial R, \mathcal{B}(\overline{R}, \rho)|_{\partial R})$ is the image of the restriction map.

By \cite{TomiyamaTakesaki61}, $\Psi$ is induced by a topological $PU_n(\mathbb{C})$-bundle equivalence of $\mathcal{B}(\overline{R}, \rho)|_{\partial R}$ with $\mathcal{B}(\overline{R}, \tau)|_{\partial R}$. Therefore, there exists an open cover $\mathcal{O}$ of $\partial R$ and family of continuous unitary-valued maps $(\bm{\mu}_U)_{U \in \mathcal{O}}$, $\bm{\mu}_U: \partial R \to U_n(\mathbb{C})$, such that the transition functions of the bundle $\mathcal{B}(\overline{R}, \tau)|_{\partial R}$ are given by conjugating the transition functions of $\mathcal{B}(\overline{R}, \rho)|_{\partial R}$ by the family $(\bm{\mu}_U)_{U \in \mathcal{O}}$. Thus we can rewrite $\Psi$: 
\[ \Psi((\sigma_U)_{U \in \mathcal{O}})= (\bm{\mu}_U\sigma_U \bm{\mu}^{-1}_U)_{U \in \mathcal{O}}  \]
for every section $(\sigma_U)_{U \in \mathcal{O}} \in \Gamma_c(\partial R,\mathcal{B}(\overline{R}, \rho)|_{\partial R})$. In particular, for every $(\sigma_U)_U$ that is the (restriction of) a continuous \emph{holomorphic} section  on $\mathcal{B}(\overline{R}, \rho)$, $(\bm{\mu}_U\sigma_U\bm{\mu}^{-1}_U)_{U \in \mathcal{O}}$ is (the restriction of) a continuous \emph{holomorphic} section in $\mathcal{B}(\overline{R}, \tau)$ . We can also apply the correspondence in Section \ref{section:constructingsections} to identify $(\sigma_U)_{U \in \mathcal{O}}$ and $(\bm{\mu}_U\sigma_U\bm{\mu}^{-1}_U)_{U \in \mathcal{O}}$ with concomitant functions with respect to the $\rho$- and $\tau$-induced actions.

The bundle $\mathcal{B}(\overline{R}, \tau)|_{\partial R}$ is determined by $\tau(1)=:\text{Ad}(D)$, $D \in U_n(\mathbb{C})$. Without loss of generality, we may assume that $D$ is diagonal: if not, then $D=PD' P^{-1}=PD'P^*$ for some diagonal $D'$ and unitary $P$, and $\Gamma_h(\overline{R}, \mathcal{B}(\overline{R}, \tau)|_{\partial R}) \simeq P \Gamma_h(\overline{R}, \mathcal{B}(\overline{R}, P^{*}\tau(\cdot)P)|_{\partial R}) P^{*}$. 

Let $\rho(1):=\text{Ad}(A), A\in U_n(\mathbb{C})$, and let $A=SA'S^*$ be a unitary diagonalization of $A$. Given a section $(\sigma'_U)_U \in \Gamma_c(\partial R, \mathcal{B}(\overline{R}, S^*\rho(\cdot)S)|_{\partial R})$, there is a section $(S^*\sigma'_US)_U \in \Gamma_c(\partial R, \mathcal{B}(\overline{R}, \rho)|_{\partial R})$, and every section in $\Gamma_c(\partial R, \mathcal{B}(\overline{R}, \rho)|_{\partial R})$ has this form.

Therefore, for every $(\sigma'_U)_U \in \Gamma_c(\partial R, \mathcal{B}(\overline{R}, S^{*}\rho(\cdot)S)|_{\partial R})$, 
\[ \Psi((S^*\sigma'_US)_{U \in \mathcal{O}})= (\bm{\mu}_US^*\sigma'_US \bm{\mu}^{-1}_U)_{U \in \mathcal{O}}  \]

Applying the correspondence between sections and concomitants, we have that for each section $(\sigma'_U)_U \in \Gamma_c(\partial R, \mathcal{B}(\overline{R}, S^{*}\rho(\cdot)S)|_{\partial R})$, the function
\begin{equation}\label{eqn:tau-concomitants} \bm{\mu}_U(e^z)S^*\sigma'_U(e^{z})S \bm{\mu}^{-1}_U(e^{z}) \text{ , } z \in \widetilde{\partial R} \end{equation}
is a concomitant with respect to the $\tau$-induced action.

Let's apply the results of Subsection \ref{section:constructingsections} to input several choices of $\sigma'$. The matrix $A'$ determining the flat bundle $\mathcal{B}(\overline{R}, S^*\rho(\cdot)S)|_{\partial R})$ is diagonal, so given $n$-tuples of complex numbers $\textbf{C}$ and $\textbf{D}$, we can construct $S^*\rho(\cdot)S$-concomitants $F_{n,\textbf{C}}$ and $F_{n,\textbf{D}}$ such that $F_{n,\textbf{D}}\circ \ln_U, F_{n,\textbf{C}}\circ \ln_U$ are in $ \Gamma_c(\partial R,\mathcal{B}(\overline{R}, S^*\rho(\cdot)S)|_{\partial R})$. Push forward the sections in (\ref{eqn:tau-concomitants}) by $\Psi$:
\begin{align*} \label{eqn:pushedforwardsections}
(\xi_{\textbf{C}})_U:=(\Psi\circ F_{n,\textbf{D}}\circ \ln_U)_{U \in \mathcal{O}} &= (\bm{\mu}_U(S^*(F_{n,\textbf{C}}\circ \ln_U)S) \bm{\mu}^{-1}_U)_{U \in \mathcal{O}}, \\
(\xi_{\textbf{D}})_U:=(\Psi \circ F_{n,\textbf{C}}\circ \ln_U)_{U \in \mathcal{O}} &= (\bm{\mu}_U(S^*(F_{n,\textbf{D}}\circ \ln_U)S) \bm{\mu}^{-1}_U)_{U \in \mathcal{O}},
    \end{align*}
which land as sections in $\Gamma_c(\partial R, \mathcal{B}(\overline{R}, \tau)|_{\partial R})$.

By previous observations at the end of Section \ref{section:constructingsections}, the cross-sections in $\Gamma_c(\partial R, \mathcal{B}(\overline{R}, \tau)|_{\partial R})$, being defined on a bundle with diagonal matrix $D$, correspond to concomitant functions $G_{\xi_\mathbf{C}} , G_{\xi_\mathbf{D}}$ whose entries are modulus automorphic functions. 
This completes step \ref{proofstep5-mapoversections} in our proof outline. 

Recall that the overall goal is to show the map that is conjugation by the family of unitaries
\begin{align*}
\boldsymbol{\bm{\mu}}_{U}(w)\coloneqq         \begin{pmatrix}\bm{\mu}_{11,U}(w)&\cdots&\bm{\mu}_{1(n-1),U}(w)&\bm{\mu}_{1n,U}(w)\\\vdots&\ddots&\bm{\mu}_{2(n-1),U}(w)&\bm{\mu}_{2n,U}(w)\\ \bm{\mu}_{(n-1)1,U}(w)&\cdots&\ddots&\vdots\\ \bm{\mu}_{n1,U}(w)&\cdots&\cdots&\bm{\mu}_{nn,U}(w)\end{pmatrix}.
\end{align*}

is locally constant, and moreover independent of $U$ so therefore pieces together to be a constant function with respect to the variable $w$. By a quick calculation we can see that a sufficient condition for conjugation by $\bm{\mu}_U$ being constant would be if (i) products of the form  $\bm{\mu}_{jk,U}(w)\overline{\bm{\mu}_{lm, U}(w)}$ are constant for all $j,k,l,m\in[n]\coloneqq\{1,2,\dots,n\}$ such that either $(j,k)=(l,m)$ or $k\neq m$, and (ii) these products are the same constant for all $U\in\mathcal{U}$.

For any fixed $n\in \mathbb{N}$, we have:

\begin{align*}
{\xi_\textbf{D}}_U &(w) = \bm{\mu}_{U}(w)F_{n,\mathbf{D}}(\ln_{U}{(w)})\bm{\mu}_{U}(w)^{*}\nonumber\\
& = \Bigg{(}\sum\limits_{\substack{q=1}}^{n}D_{q}\bm{\mu}_{jq, U}(w) \overline{\bm{\mu}_{kq,U}(w)} \Bigg{)}_{j,k\in[n]} \\
\text{and} & \\
{\xi_\textbf{C}}_U & (w) = \bm{\mu}_{U}(w)F_{n,\mathbf{C}}(\ln_{U}{(w)})\bm{\mu}_{U}(w)^{*}\nonumber\\
& = \Bigg{(}\sum\limits_{\substack{q=1\\m=1\\q< m}}^{n}C_{qm}\bm{\mu}_{jq, U}(w)\overline{\bm{\mu}_{km,U}(w)}e^{K_{qm}\text{ln}_{U}(w)} + \sum\limits_{\substack{q=1\\m=1\\ m < q}}^{n}C_{qm}\bm{\mu}_{jq,U}(w)\overline{\bm{\mu}_{km,U}(w)}e^{-K_{mq}\text{ln}_{U}(w)} \Bigg{)}_{j,k\in[n]}.\\
\end{align*}

Now, when we set $\textbf{D}= (0, \ldots, 0, D_k, 0, \ldots 0)$, we get that for each pair $j,k \in [n]$, the expression $(\bm{\mu}_{jk,U}(w)\overline{\bm{\mu}_{jk,U}(w)})_U = (|\bm{\mu}_{jk,U}(w)|^2)_U$ is the $(j,j)$ entry of ${\xi_\textbf{D}}_U(w)$. 
The correspondence between sections and concomitants gives us that this entry $|\bm{\mu}_{jk,U}(w)|^2$ defines a modulus automorphic function $ \tilde{f}_{jk} : \partial \widetilde{D} \to \mathbb{C}$. In particular, whenever $e^z \in U \cap V$, $|\bm{\mu}_{jk,U}(e^z)|^2 = |\bm{\mu}_{jk,V}(e^z)|^2 = \tilde{f}_{jk}(z)$ and the value of $|\bm{\mu}_{jk,U}(e^z)|^2$ is independent of $U$. 
Moreover, $\tilde{f}_{jk}$ is both real-valued on $\partial \widetilde{D}$ and the restriction of a function analytic on $D$. By a standard argument using the function theory on the universal cover $\widetilde{D}$, we conclude that $\bm{\mu}_{jk,U}(w)\overline{\bm{\mu}_{jk,U}(w)}=|\bm{\mu}_{jk}(w)|^2$ on $\partial R$ is constant.

We next apply a slightly more complex version of this argument to the case of $\bm{\mu}_{jk,U} \overline{\bm{\mu}_{lm,U}}$ where $k \neq m$. Let $A'\coloneqq \text{diag}(a_1,\dots, a_n)$ be the unitary matrix determining $S^*\rho(1)S$. We can then define the concomitant function $F_{n,\textbf{C}}(z)$ from section \ref{section:constructingsections}.

As in the previous case, we can map the section $F_{n,\mathbf{C}}(\ln_U(w))$ over to a section of the bundle $\mathcal{B}(\overline{R},\tau)|_{\partial R}$ using $\Psi$:

\begin{align*}
\Psi & (F_{n,\mathbf{C}}(\ln_{U}{(w)})) = \boldsymbol{\bm{\mu}}_{U}(w)F_{n,\mathbf{C}}(\ln_{U}{(w)})\boldsymbol{\bm{\mu}}_{U}(w)^{*}\nonumber\\
&  = \Bigg{(}\sum\limits_{\substack{q=1\\m=1\\q< m}}^{n}C_{qm}\bm{\mu}_{jq, U}(w)\overline{\bm{\mu}_{km,U}(w)}e^{K_{qm}\text{ln}_{U}(w)} + \sum\limits_{\substack{q=1\\m=1\\ m < q}}^{n}C_{qm}\bm{\mu}_{jq,U}(w)\overline{\bm{\mu}_{km,U}(w)}e^{-K_{mq}\text{ln}_{U}(w)} \Bigg{)}_{j,k\in[n]}
\end{align*}
For each choice of $(C_{qm})_{qm}$, the $(j,k)$ entry of the matrix in the above equation
corresponds to a modulus automorphic function on $\partial \widetilde{D}$. Consider the choices: 

\begin{alignat*}{2}
& C_{km}=1, k<m; C_{rs}=0 \text{ else}; \text{ entry } (j,l): &&\qquad \tilde{g}_{km}(z):= \bm{\mu}_{jk,U}(e^z)\overline{\bm{\mu}_{lm,U}(e^z)}e^{K_{km}z}\\
& C_{km}=1, k>m; C_{rs}=0 \text{ else}; \text{ entry } (j,l): &&\qquad \tilde{h}_{mk}(z):=\bm{\mu}_{jk,U}(e^z)\overline{\bm{\mu}_{lm,U}(e^z)}e^{-K_{mk}z}
\end{alignat*}

The functions $\tilde{g}_{km},\tilde{h}_{mk}$ also must be restrictions to $\partial \widetilde{D}$ of functions that are analytic on $D$. 

Now, $e^{\pm K_{ij}z}$ are already modulus automorphic functions, so the products $\tilde{g}_{km}(z)e^{-K_{km}z}$ and $\tilde{h}_{mk}(z)e^{K_{mk}z}$ are also modulus automorphic (with a different multiplier), and will still be restrictions of functions that are analytic on $D$. Thus for $k\neq m$, we still have that $\bm{\mu}_{jk,U}(e^z) \overline{\bm{\mu}_{lm,U}(e^z)}$ are analytic as well. Since the order of the pairs $(j,k)$ and $(l,m)$ doesn't matter, we also have that the conjugate function $\overline{\bm{\mu}_{jk,U}(e^z)} \bm{\mu}_{lm,U}(e^z)$ is also a modulus automorphic function that is the restriction of a function analytic on $D$. We can again use a standard argument on $\widetilde{D}$ to conclude that $\bm{\mu}_{jk}(w) \overline{\bm{\mu}_{lm}(w)}$ is a constant function on $\partial R$.
    \end{proof}

\begin{cor} The algebras $\Gamma_h(\overline{R},\mathcal{B}(\overline{R},\rho))$ over the annulus are classified, up to complete isometric isomorphism, by restricted flat bundle equivalence.
    \end{cor}

\section{Generalization to other domains in the commuting case} \label{sec:GeneralizationToOtherDomains}

In our main theorem, $\overline{R}$ is an annulus in the complex plane. This mostly comes into play when analyzing the representation $\rho : \pi_1(\overline{R}) \to PU_n(\mathbb{C})$; for an annulus the fundamental group $\pi_1(\overline{R})$ is isomorphic to $\mathbb{Z}$, thus $\rho$ is determined by $\rho(1)$. Suppose instead we consider a more general finitely-bordered and smoothly-bordered planar domain $R$ with $b$ boundary components. We may still consider complex functions on $R$, and the algebras $A(\overline{R})$, $\Gamma_h(\overline{R},\mathcal{B}(\overline{R}))$, etc., like in \cite{McCormick19}.

For such a bordered planar domain, $\pi_1(\overline{R})$ is a free group on $b-1$ generators, $\mathbb{F}_{b-1}$. Therefore, $\rho$ will be determined by where $\rho$ sends the generators of $\mathbb{F}_{b-1}$. As before, given a choice of base points and paths one can build the $\rho$-induced representation of $\Pi_1(\partial R)$. 

Understanding when the $\rho$- and $\rho '$-induced representation are unitarily conjugate will be, in the most general case, a complex problem. However, this problem can be simplified if one assumes that the generators are sent to elements in $PU_n(\mathbb{C})$ that can be represented as conjugation by \emph{commuting} matrices $A_1, \ldots, A_{b-1} \in U_n(\mathbb{C})$. Here we will briefly sketch how one can rework the calculations in the previous section to treat this case.

First note that the universal covering space of $\overline{R}$ can still be viewed as a subset of the closed disk $\overline{D}$, which we continue to call $\widetilde{D}$. Moreover, the Lebesgue measure of the boundary inside the boundary circle is still $2\pi$ \cite[Ch.~XI]{Tsuji1975}. We will also still be able to apply the remark after Proposition \ref{mainthmmod}. 

Let us begin with Step 1 in our proof of Proposition \ref{mainthmmod}. Because of our commuting assumption, the matrices $A_1, \ldots, A_{b-1}$ can be simultaneously diagonalized by a unitary. Thus we may relate the sections of a bundle $\mathcal{B}(\overline{R},\rho)$ with (conjugated) sections of a bundle $\mathcal{B}(\overline{R}, \rho')$ where $\rho'$ takes the generators of $\pi_1(\overline{R})$ to diagonal unitaries $A_1', A_2', \ldots, A_{b-1}' $. 

To generate concomitants (and thus sections) with respect to $\mathcal{B}(\overline{R}, \rho')$ we need to consider analogues of the functions $F_{\bm{C}}, F_{\bm{D}}: \widetilde{D} \to M_n(\mathbb{C})$. Replace the invertible modulus automorphic functions built by the exponentials on the infinite strip by continuous holomorphic functions $f_i : \widetilde{D} \to \mathbb{C}$ satisfying $f(g_j \cdot z)= u_{ij} \cdot f(z)$ where $|u_{ij}|=1$ and the value of $u_{ij}$ depends on the eigenvalues of the representing matrix $A_j'$ for the generator $g_j \in \mathbb{F}_{b-1}$.

Next, we recall that even in the more general setting of a finitely bordered planar domain $\overline{R}$, we still may study the modulus automorphic functions \cite{AbrahamseDouglas76}. Making these adjustments, we may follow the proofs in Section \ref{section:results} through to conclude that bundles $\mathcal{B}(\overline{R},\rho)|_{\partial R}$ and $\mathcal{B}(\overline{R},\tau)|_{\partial R}$ coming from $\rho,\tau$ with commuting generator images are restricted flat equivalent if and only if the algebras $\Gamma_h(\overline{R},\mathcal{B}(\overline{R},\rho))$ and $\Gamma_h(\overline{R},\mathcal{B}(\overline{R},\tau))$ are completely isometrically isomorphic.

\section{Acknowledgements}
Part of this work began at the University of Iowa as a part of the second-named author's PhD thesis. The first-named author was supported by a California State University Long Beach student research assistantship fellowship made possible by Richard D.~Green funding during the summer and fall of 2021. 

\bibliography{refs.bib}
\bibliographystyle{plain}

\end{document}